\documentclass[12pt]{amsart}
\usepackage{bm,amsmath,amssymb,eucal}

\usepackage{hyperref}

\makeatletter
\@namedef{subjclassname@2020}{%
  \textup{2020} Mathematics Subject Classification}
\makeatother

\newtheorem{Th}{Theorem}[section]
\newtheorem{Prop}[Th]{Proposition}
\newtheorem{Lem}[Th]{Lemma}
\newtheorem{Cor}[Th]{Corollary}
\newtheorem{Fact}[Th]{Fact}

\theoremstyle{definition}

\newtheorem{Pro}[Th]{Problem}

\newtheorem{Ex}[Th]{Example}

\theoremstyle{remark}
\newtheorem{Rem}[Th]{Remark}
\newtheorem{Rems}[Th]{Remarks}

\newcommand\SE{\mathcal E}

\newcommand\SH{\mathcal H}
\newcommand\SI{\mathcal I}

\newcommand\SN{\mathcal N}

\newcommand\SL{\mathcal L}

\newcommand\N{\mathbb N}

\newcommand\K{\mathbb K}

\newcommand\bbm{\boldsymbol m}
\newcommand\bbn{\boldsymbol n}

\newcommand\ep{\varepsilon}
\newcommand\Si{\Sigma}
\newcommand\si{\sigma}
\newcommand\ga{\gamma}
\newcommand\mub{\bar\mu}

\newcommand\Ga{\Gamma}

\newcommand\la{\lambda}
\newcommand\f{\varphi}

\newcommand\ro{\rho}
\newcommand\dl{\delta}
\newcommand\om{\omega}

\newcommand\sm{\smallsetminus}
\newcommand\sbs{\subset}

\newcommand\leqs{\leqslant}
\newcommand\geqs{\geqslant}

\newcommand\oo{\infty}
\newcommand\loo{\ell_{\oo}}

\newcommand\ovl{\overline}

\newcommand\conv{\operatorname{conv}}

\newcommand\nor{\|{\cdot}\|}

\newcommand\mueq{\overset{\mu}{=}}
\newcommand\musbs{\overset{\mu}{\sbs}}

\begin{document}
\title{Control Measures for Bochner~$L_0$-valued Vector Measures}

\author[L.~Drewnowski]{Lech Drewnowski \textdagger } 
\address{Faculty of Mathematics and Computer Science\\
A.~Mickiewicz University\\
Uniw.Pozna\'nskiego 4, 61--614 Pozna\'n, Poland}
\email{drewlech@amu.edu.pl}

\author[A.~Teixeira]{Alexandre R. Teixeira}
\address{Department of Statistics\\
State University of Campinas (UNICAMP)\\
Rua Sérgio Buarque de Holanda, 651| 13083-859, Campinas, SP, Brazil}
\email{riccimath123@gmail.com;a16307@dac.unicamp.br}
\date\today

\footnotetext[1]{\textdagger Lech Drewnowski passed away on November 2023. Both authors contributed equally to the ideas in the text.}

\begin{abstract}
It is shown that for any finite positive measure $\mu$ defined on a measure space $(S, \Sigma)$, and any Banach or Fr\'echet space $Z$, the control measure Theorem of Talagrand (\cite{T}) is true for the case when the (stochastic) vector measure $\bbm:\SE\to L_0(\mu,Z)$, defined on another measurable space $(E, \SE)$, takes values in $L_{0}(\mu,Z)$, the Bochner space of vector-valued functions associated to $\mu$ and $Z$. As a consequence, we also obtain a Rybakov type result for this control. Finally, we give the relation of this result to bounded multiplier properties (BMP) of $F$-spaces and pose various open problems related to it.
\end{abstract}

\maketitle

\section{Introduction}
Our main aim with this short note is to formulate, in rigorous terms, and prove the result cited in the Abstract, which appears below in a somewhat refined form as Theorem \ref{theoremcontrol} and Theorem \ref{rybakovtheorem}. This work was motivated, at least in part, to contribute further to the relation between functional analytic properties of vector measures and their applications in stochastic integration, represented in the classical literature of stochastic analysis by Schwartz (\cite{schwartz2006semi}) and Metivier \& Pellaumail (\cite{metivierpellaumail}) works on vector measures induced by (formal) semimartingales, as well as by the more recently publications of M. M. Rao on his massive volume on random and vector spaces (see \cite{rao}) and Basse-O'Connor and his coauthors in \cite{basse2014stochastic}, which presents new ideas related to stochastic integration on the real line. These two results, control measure and Rybakov choice of control, are, as we shall comment, related to even more general, though similar in certain aspects, summability properties on $F$-spaces with the BMP (\textit{bounded multiplier property}), hence, naturally, also in vector measures taking values in spaces of this type (see \cite{M-P}, \cite{RNW}, \cite{T}, \cite{KPR}, \cite{LD-4}), and in the integration (in the Bartle-Dunford-Schwartz sense) of scalar functions with respect to such measures (see\cite[p. 120 and ff.]{Rol} and esp.\cite{LDIL-3}). We present our main result in Section \ref{sec:control}, with Section \ref{sec:basic} and Appendix \ref{appendix} serving as a guide to fix some notations and expose some classical results for spaces of measurable functions with values in Banach spaces that we shall need later. With the intentions of giving the readers a better picture of this research area, and, at the same time, inviting them to the work on  still open problems that are worthy of attention, we shall present various open questions throughout the text, and comment on their significance when appropriate. For the general unexplained linear-topological notions the reader is referred to the classical treatise \cite{J}, and for those involving submeasures - the paper of the first author \cite{LD-trs3} and its earlier parts.
\section{Basic notions and notation}\label{sec:basic}

We shall fix some essential notations and concepts to be used in this work for the remainder of its sections. 

Throughout, we assume that:
\begin{itemize}
\item\quad $(S,\Si,\mu)$ is a finite positive measure space.

\item\quad $Z$ is a Banach (or, more generally, Fr\'echet) space, with a norm or an $F$-norm $\nor$ defining its topology $\tau$.
    \end{itemize}

    Any additional requirements on $\mu$ or $Z$ will be imposed explicitly.

In this setting, the topology $\tau_\mu$ of convergence in measure $\mu$ in a  linear space $F$ of functions $f:S\to Z$ is the topology generated by the generic balanced $\tau_\mu$-neighborhoods of zero in $F$:
$$
U_\mu(\dl):=\{f|\exists A:\mu(A)<\dl\ \&\  f(A^c)\sbs U\},
$$
, $A^c$ denoting the complement of $A$ in $S$. Notice that no sort of measurability of functions $f$ is required here. $\tau_\mu$ can be defined, for instance, by the group-seminorm
$$
\|f\|_\mu = \inf\{\ep > 0: \mub(\|f(\cdot)\|\geqs \ep)\leqs\ep\},
$$
where the submeasure $\mub$, the \textit{Jordan extension} of $\mu$, is given by $$\mub(A):=\inf\{\mu(B):A\sbs B\in\Si\}\quad\text{for $A\sbs S$}.
$$
Notation $|\cdot|_\mu$ will be used (though not too often) when $Z=(\K,|\cdot|)$.
To be strict: $\tau_\mu$ is a genuine linear topology, and $\nor_\mu$  a genuine $F$-seminorm when each function $f\in F$ is bounded outside a `$\mu$-small' set $A$, that is,
$$
\forall\dl>0\exists A\in\Si:\mu(A)<\dl\quad  \text{and  $f(A^c)$ is bounded in $Z$}.
$$
It is surely so for $F=\SL_0(\mu,Z)$ below, as follows easily from that, by definition, it is the $\tau_\mu$-closure of simple functions, trivially satisfying the above condition.

\subsection{ More notions, notation and conventions}
 $\SN(\mu)$ denotes the $\si$-ideal of $\mu$-null sets; their subsets are called \textit{$\mu$-negligible}. The symbols $\mueq$ and $\musbs$ will sometimes be used for the equality and containment $\mu$-a.e. (or for $\mu$-a.a. $s\in S$), resp. Any set $A\in\Si$ is tacitly considered with its induced $\si$-algebra $\Si_A=\{B\in\Si:B\sbs A\}$, esp. if the context suggests so. Finally, recall that a function $f:S\to Z$ is said to be \textit{Bochner measurable}, or \textit{$\mu$-measurable} if a sequence of $\Si$-simple $Z$-valued functions converges to $f$ $\mu$-a.e. on $S$ (note that convergence in (sub)measure $\mu$ may be used here as well). It is clear that left compositions with continuous mappings, in particular - linear operators, preserve this type of measurability. We will occasionally write $\SL_0(S,\Si,\mu;Z)$ for the linear space of all such functions $f:S\to Z$, but shorter notation $\SL_0(\mu,Z)$ or $\SL_0(S,Z)$ will also be used. However, what really matters to us, is the quotient linear space $L_0(\mu,Z)=L_0(S,Z)=L_0(S,\Si,\mu;Z)$ of all $\mu$-equivalence classes $[f]=[f]_Z$. If $Z=\K$, the field of scalars of $Z$, we denote this space by $L_0(\mu)$ or $L_0(S)$.
$\tau_\mu$ will stand for the usual (or \textit{strong}) \textit{topology of convergence in (sub)measure} $\mu$, corresponding to the default strong (or ($F$) norm) topology of $Z$; $\nor_\mu$ - for the preferred (e.g., one displayed above), $F$-seminorm or its `quotient' $F$-norm defining that topology in $\SL_0$ or $L_0$, resp. It is well known that $L_0(\mu,Z)$ is an $F$-space.

$\tau^w_\mu$ will denote  the \textit{topology of weak convergence in (sub)measure} $\mu$, the one corresponding, in accordance  with the general concept explained above, to the weak topology $\tau^w =\si(Z,Z^*)$ of $Z$. Like $\tau_\mu$, it is Hausdorff (see Fact \ref{fact:tau*H} below). Obviously, ($f\to |z^*f|_\mu:z^*\in Z^*)$ is a family of $F$-seminorms defining it. It is also clear that  $\tau^w_\mu$ is weaker than $\tau_\mu$. In general, this relation is strict. To see this, consider the sequence $(e_n)$ of standard unit vectors in $c_0$, and the sequence $(e_n 1_{[0,1]})$ in $L_0[0,1]$. A similar example can be provided for any Banach space $Z$ without the Schur property.
\subsection{Preliminary facts}
The two facts right below are well known. For the first, see e.g., \cite{Dunf}, \cite{D-U}.
A proof of the second is adduced for the readers' convenience.
\begin{Fact}\label{fact:b-b}
A function $f:S\to Z$ is Bochner measurable iff there exists $N\in\SN(\mu)$ such that $f|N^c$ is separably valued and Borel measurable. Note: the latter remains valid if $Z$ is replaced by any of its subsets containing $f(N^c)$.
\end{Fact}

\begin{Fact}
For every separable subset $E$ of $Z$ there is a sequence $(z^*_n)$ in $Z^*$ that is total on $E$, i.e., whenever $0\ne z\in E$, then $z^*_n(z)\ne 0$ for some $n$.
\end{Fact}
\begin{proof}Assume first that $Z$ is Banach.
Select any sequence $(z_n)$ dense in $E$. By a corollary to the Hahn-Banach theorem, there is a sequence $(z^*_n)$ of norm one elements of $Z^*$ such that $z^*_n(z_n)=\|z_n\|$ for each $n$. Now, take any $0\ne z\in E$, and next choose $n$ so that $\|z-z_n\|<\|z\|/2$. Note that $\|z_n\|\geqs \|z\|-\|z-z_n\|>\|z\|/2$. Thus next, $|z^*_n(z)|\geqs |z^*_n(z_n)|-|z^*_n(z-z_n)|\geqs \|z_n\|-\|z-z_n\|>0$. That's it!. When $Z$ is Fr\'echet, embed it in a countable product of Banach spaces and consider the coordinate projections of $E$.
\end{proof}
\begin{Fact}\label{fact:tau*H}
The topology $\tau^w_\mu$ in $L_0(\mu,Z)$ is Hausdorff.
\end{Fact}
\begin{proof}
For, let $f\in L_0(\mu,Z)$ be such that $z^*f\mueq 0$ for each $z^*\in Z^*$.
By Fact \ref{fact:b-b} there is a separable set $E\sbs Z$ such that $f(S)\musbs E$. Choose a sequence $(z^*_n)$ in $Z^*$ that is total on $E$. Now, if
the support of $f$ were not $\mu$-negligible, neither would be that of $z^*_n f$ for some $n$, contrary to our assumption.
\end{proof}
\subsection[Closed subspaces]{Closed subspaces of $Z$ and their counterparts in $L_0(\mu,Z)$}
Let $X$ be a closed subspace of $Z$. Of course, $\SL_0(\mu,X)$ is a  subspace of $\SL_0(\mu,Z)$. However, in order to view $L_0(\mu,X)=\{[f]_X:f\in\SL_0(\mu,X)\}$ as a genuine subspace of $L_0(\mu,Z)$, one should redefine the former, with a slight change in notation (although still provisional), as $L_0(\mu,X\sbs Z):=\{[f]_Z:f\in\SL_0(\mu,X)\}$.
Apart from these formalities, the correspondence $[f]_X \leftrightarrow [f]_Z$ is isometric for the $F$-norms $\nor_\mu$. The lemma below clarifies the situation.

\begin{Lem}
 A function $g\in
\SL_0(\mu, Z)$ is $\mu$-equivalent to some function $f\in \SL_0(\mu,X)$ iff $g(S)\musbs X$.
\end{Lem}
\begin{proof}

 `if': Consider the set $N=\{s: g(s)\notin X\}\in\SN(\mu)$ and define $f=1_N\cdot 0+1_{S\sm N}\cdot g$. Using relevant definitions or Fact \ref{fact:b-b}, one sees that $f$ is as required.
\end{proof}
\begin{Prop}\label{prop:tau*-closed}
$L_0(\mu,X\sbs Z)$ is a $\tau^w_\mu$-closed (and a fortiori $\tau_\mu$-closed) subspace of $L_0(\mu,Z)$.
   \end{Prop}
    \begin{proof} Consider the operator $Q:L_0(\mu,Z)\to L_0(\mu,Z/X);f\mapsto qf$. Clearly, it is continuous for either of topologies $\tau_\mu$ and $\tau^w_\mu$ in both the domain and the range. As for the latter, simply notice that for every $\xi\in (Z/X)^*$, setting $x^*:=\xi q\in X^*$, one has $\xi Qf=x^*f$. In consequence, and since $\tau^w_\mu$ is Hausdorff, $L_0(\mu,X\sbs Z)=Q^{-1}(\{0\})$ is as asserted.
   \end{proof}
  
\section[The existence of  control measures]{The existence of  control measures for $L_0(\mu,Z)$-valued vector measures}\label{sec:control}

In what follows, we let $(E,\SE)$ be a measurable space (i.e, as above, $E$ is a set; $\SE$ a $\si$-algebra of its subsets), and $\bbm:\SE\to L_0(\mu,Z)$ be a (countably additive) vector measure. Comparing the contents of  \cite{T} and \cite{LD-4}, both related to controls and boundedness of the aforementioned class of vector measures (the first on the real case $E=\mathbb{R}$), we are interested here in the following question, which was left open by \cite{LD-4} after proving some boundedness results:

    \begin{Pro}\label{pro:control}
    Does $\bbm$ always admit  a \textit{control measure}, that is, a finite positive (or, equivalently in our, context probability) measure $\nu$ on $\SE$ such that $\bbm(A)=0$ whenever $\nu(A)=0$?
    \end{Pro}

    As first strategy of proof, one may ask if it is possible to reduce the problem above to one of the main results of \cite[Th. B]{T}:

   {\rm (T)}\quad   \textit{If $Z=\K$, a control measure for $\bbm$ does exist.}

   In fact, using (T) and some results and ideas from
   \cite[Sec. 10]{LD-trs3}, we may provide an affirmative answer to problem \ref{pro:control}. We begin this task with an auxiliary result from \cite[Th. 10.4]{LD-trs3}; also see \cite[Th. 4.8]{LD-contr}:

   \begin{Prop} \label{prop:countable-uoc}
   Every set $H$ of uniformly o.c. submeasures on $\SE$ has a countable subset $K$ such that the. o.c. submeasures $\eta_H=\sup H$ and $\eta_K=\sup K$ are equivalent, that is, $\eta_H\ll\eta_K\ll\eta_H$ or, which means the same, $\SN(\eta_H)=\SN(\eta_K)$
   \end{Prop}

   With this preliminary, we can exhibit the following solution to \ref{pro:control}:

   \begin{Th} [Controls for $L^{0}$ Böchner-Valued Vector Measures] \label{theoremcontrol} Let $Z$ be a Banach space and $\bbm:\SE\to L_0(\mu,Z)$ be a countably additive vector measure. Then, $\bbm$ has a control measure.
   \end{Th}

   \begin{proof}
       Let $B_{*}$ the closed unit ball of $Z^{*}$. Notice that, since $|z^*\bbm(A)|_\mu\le\|\bbm(A)\|_\mu$ for all $z^*\in B_{*}$, $A\in\SE$, denoting the submeasure majorants of $z^*\bbm$ and $\bbm$ by $\ovl{z^*\bbm}$ and $\ovl{\bbm}$ respectively, one has that:

       \[\ovl{z^*\bbm}\le\ovl{\bbm},\]

       for each $z^*\in B_{*}$. As $\ovl{\bbm}$ is o.c., it follows by the inequality above that the submeasures in the set $H=\{\ovl{z^*\bbm}:z^*\in B_{*}\}$ are uniformly o.c. Therefore, we may apply Proposition \ref{prop:countable-uoc} to get a sequence $(z^*_n)$ in $B_{*}$ such that the associated submeasures with it subset $K$ of $H$ satisfies the properties required in Proposition \ref{prop:countable-uoc}. Finally, with this extracted sequence construct, as in the proof of Proposition \ref{const.of.measure} right below, a probability measure $\nu$ on  $\SE$ that controls the submeasures obtained by \ref{const.of.measure}. Then whenever $\nu(A)=0$, one has $\eta_K(A)=0$ so that also $\eta_H(A)=0$ for all $A \in \SE$. Consequently, $z^*\bbm(A)\mueq 0$ for all $z^*\in B_{*}$ and hence $\bbm(A)\mueq 0$ for all $A \in \SE$. Thus $\nu$ is a control measure for $\bbm$.
   \end{proof}

   \begin{Rem}
   For the case when $Z$ is a Fr\'echet space, embedding it into a countable product of Banach spaces, and applying the previous case to the composition of $\bbm$ with each coordinate projection, and combining the results in an already familiar manner, we get
   \end{Rem}

   \begin{Th} 
   
   Let $Z$ be a Fréchet space and $\bbm:\SE\to L_0(\mu,Z)$ be a countably additive vector measure. Then, $\bbm$ has a control measure.
   \end{Th}

   With this result in mind, we would like to know if a little more can be added to the above in the Banach case, in particular, if a Rybakov-type result for $\bbm$ is valid. In fact, we shall prove that:
   
\begin{Th} [Rybakov Theorem] \label{rybakovtheorem}
    There exists $z^*_0\in B_{*}$ such that $\bbm\ll z^*_0\bbm$.

   In other words, any control measure for $z^*_0\bbm$, existing by (T), is such also for $\bbm$.
\end{Th} 

To obtain such result, we shall follow the strategy, in the Banach case of \cite{LD-trs3}, which doesn't use any type of Radon-Nikodym arguments, which is needed here as we don't have any density (of RN type) at our disposal. For this, we need some preliminary results that have to be adapted from \cite{LD-trs3}, but not straightforwardly so. To no interrupt the flow of the text, we collect the necessary material in Appendix \ref{appendix} \footnote{Even though the results are valid for general F-spaces, we state and prove them only for $L_{0}(\mu)$.} Also, before presenting the proof of the last theorem, we begin with two easy consequences of (T).

\begin{Prop} \label{const.of.measure}
If $Z$ is separable or, more generally, $Z$ is an $F$-space and there is a total sequence $(z^*_n)\sbs Z^*$, then $\bbm$ has a control measure $\nu$.
\end{Prop}
\begin{proof}
By (T), for every $n$ the vector measure $z^*_n \bbm$ admits a control probability measure $\nu_n$. To finish, just set $\nu=\sum_n 2^{-n}\nu_n$ (cf. the proof of Fact \ref{fact:tau*H}).
\end{proof}

\begin{Cor}
If the vector measure $\bbm$ has a separable range, then it admits a control measure.
\end{Cor}
\begin{proof}
By the assumption, there is a sequence $(A_n)\sbs \SE$ such that the sequence $ ( \bbm(A_n))$ is dense in $\bbm(\SE)$. Next, by Fact \ref{fact:b-b}, we can find a closed separable subspace $X$ of $Z$ such that for some $T\in\Si$ with $\mu(S\sm T)=0$ one has $\bbm(A_n)|T \in L_0(\mu|T,X)$ for each $n$. From this, in view of Proposition \ref{prop:tau*-closed}, it follows that $A\mapsto \bbm(A)|T$ is a vector measure from $\SE$ to $L_0(\mu|T,X)$. By the preceding Proposition \ref{const.of.measure}, it admits a control measure which is clearly such for the original vector measure $\bbm$ as well.
\end{proof}

We now return to the Rybakov theorem.

\begin{proof}[Proof of the Rybakov theorem]
For each $z^*\in Z^*$ let
$$
\mu_{z^*}:=z^*\bbm:\SE\to L_0(\mu).
$$
Denote by $\ovl{\mu_{z^*}}$ its semivariation in the sense of \cite[Sec.\ 10]{LD-trs3}. Since
$$
|\mu_{z^*}(A)|_\mu=|z^*\bbm(A)|_\mu\le \|z^*\|\,\|\bbm(A)\|_\mu\qquad(A\in\SE),
$$
we have
$$
\ovl{\mu_{z^*}}\le \|z^*\|\,\ovl{\bbm}.
$$
As $\ovl{\bbm}$ is order continuous, the family
$$
\SH:=\{\ovl{\mu_{z^*}}:z^*\in B_*\}
$$
is uniformly order continuous. Moreover, every scalar $L_0(\mu)$-valued countably additive vector measure has bounded range by \cite{T} (see also \cite{KPR}). We may therefore apply Theorem \ref{thm:drewnowski-L0} below to the family $(\mu_{z^*})_{z^*\in B_*}$. Thus there are a sequence $(z_n^*)\sbs B_*$ and scalars $(c_n)\in \ell_1$ such that the series
$$
\lambda:=\sum_{n=1}^{\infty} c_n\mu_{z_n^*}
$$
converges uniformly on $\SE$, defines a countably additive measure $\lambda:\SE\to L_0(\mu)$, and satisfies
$$
\ovl{\lambda}(A)=0\iff \sup_{z^*\in B_*}\ovl{\mu_{z^*}}(A)=0\qquad(A\in\SE).
$$
Since $(c_n)\in\ell_1$ and $\|z_n^*\|\le 1$ for all $n$, the series $\sum_n c_n z_n^*$ converges absolutely in $Z^*$. Let
$$
z_0^*:=\sum_{n=1}^{\infty} c_n z_n^*\in Z^*.
$$
For each $A\in\SE$ we then have
$$
\lambda(A)=\sum_{n=1}^{\infty} c_n z_n^*\bbm(A)=z_0^*\bbm(A),
$$
so that $\lambda=z_0^*\bbm$.

Now let $A\in\SE$ be such that $z_0^*\bbm(A)=0$. Then $\ovl{\lambda}(A)=0$, hence
$$
\sup_{z^*\in B_*}\ovl{\mu_{z^*}}(A)=0.
$$
Therefore $\ovl{\mu_{z^*}}(A)=0$ for every $z^*\in B_*$. In particular,
$$
z^*\bbm(A)=0\qquad(z^*\in B_*).
$$
By homogeneity, the same equality holds for every $z^*\in Z^*$. Thus $z^*\bbm(A)\mueq 0$ for all $z^*\in Z^*$, and Fact \ref{fact:tau*H} implies that $\bbm(A)=0$ in $L_0(\mu,Z)$. We have proved that $\bbm\ll z_0^*\bbm$.

If $z_0^*=0$, then $\bbm=0$ and the conclusion is trivial. Otherwise, replacing $z_0^*$ by $z_0^*/\max\{1,\|z_0^*\|\}$, we obtain an equivalent functional still denoted by $z_0^*$ and belonging to $B_*$.
\end{proof}

With the proof now in place, we make some further methodological comments regarding the proof of the control result in the context of (T).

\begin{Rem}
Talagrand, to get (T), used first the fact that $\conv\bbm(\SE)$ is bounded in $L_0(\mu)$ so that the integration operator $I_{\bbm}:\f\mapsto\int_E\f\,d\bbm$ from the Banach space $L_\oo(E,\SE)$ of bounded scalar Borel measurable functions to $L_0(\mu)$ can be defined and is continuous. Next, appealing to a result of B. Maurey \cite{M}, he concluded that   $\bbm= \psi\cdot \bbn$, where $\bbn$ is a vector measure from $\SE$ to the Hilbert space $L_2(\mu)$, and $0\le \psi\in L_0(\mu)$. Hence from a classical 1955 result of Bartle, Dunford and Schwartz it followed that $\bbn$ admits a control measure $\nu$, which is obviously such for $\bbm$, too.
Now, a question arises, can we possibly repeat this reasoning, when $Z$ is an arbitrary Banach space? The first of its steps \underline{can} be repeated, because $\conv\bbm(\SE)$ is still bounded in $L_0(\mu,Z)$ (cf. the next section) and, therefore, one may rightly consider the integration operator $I_{\bbm}: L_\oo(E,\SE)\to L_0(\mu,Z)$. But what about a suitable for us more general version of  Maurey's result? This point is as yet unclear to the authors. A likely substitute for $L_2(\mu)$ seems to be the \textit{weak Hilbert space}:
   $$
   \text{\rm w-}L_2(\mu,Z) :=\{f\in L_0(\mu,Z) | \forall z^*\in Z^*:z^*f\in L_2(\mu)\},
   $$
   with the norm
   $$
   \text{\rm w-}\|f\|_2:=\sup_{\|z^*\|\le 1} \|z^*f\|_2.
   $$
   Further, it is reasonable to expect a factorization of $\bbm$ of the form
   $\bbm=\psi\cdot\bbn$ for some $0\le\psi\in L_0(\mu)$ and a vector measure
   $\bbn:\SE\to \text{\rm w-}L_2(\mu,Z)$.
\end{Rem}
\begin{Rem}
If, in Problem \ref{pro:control}, one allows $\mu$ to be an o.c. submeasure, and $Z$ to be an $F$-space, then  necessary requirements for an affirmative answer are that $\mu$ is equivalent to a finite positive measure (i.e., both have the same null sets), and that every countably additive vector measure $m:\SE \to Z$ admits a control measure. To see the first, consider $\bbm:\Si\to L_0(\mu,Z);A\mapsto z1_A$ for a fixed $0\ne z\in Z$. For the second, use $\bbm:\SE\to L_0(\mu,Z);A\mapsto m(A)1_S$. As is by now well known to everyone thanks to M. Talangrand's work \cite{T2}, the first requirement is \underline{not always} satisfied.
\end{Rem}
\subsection{F-spaces with BMP and  convexly bounded vector measures}\label{sec:BMP}
   For a more detailed information related to this section, see \cite[Sec. 2.2]{LDIL-3}. Recall that a topological linear space $F$ is said to have the \textit{Bounded Multiplier Property} (BMP) if every subseries convergent series $\sum_n z_n$ in $F$ is \textit{bounded multiplier convergent}. that is, the series $\sum_n a_n z_n$ converges for each sequence $(a_n)\in\loo$. It was shown in \cite{M-P} and \cite{RNW}, resp., that
   \begin{Fact}
    For $\mu$ a finite positive measure, the $F$-spaces $L_0(\mu)$ and $L_0(\mu,Z)$ have the {\rm BMP}.
     \end{Fact}

     In the context of this fact, the following question arises: 
     
     \begin{Pro}\label{pro:BMP}
     Is the above still true when $\mu$ is an o.c. submeasure and/or $Z$ is an $F$-space with the {\rm BMP}?
     \end{Pro}
     The following well known characterization may be useful in attacking the problem, by focusing attention on condition (bmp) stated below instead of {\rm BMP}. Anyway, this approach  works almost trivially when $L_0(\mu,Z)$ is replaced by the much simpler complete metrizable topological vector group $F(S,Z)$ of all functions from $S$ to $Z$ with the topology of uniform convergence. Alas, in the case of interest to us, things don't go that smoothly.
   \begin{Fact}
   An $F$-space $F$ has the {\rm BMP} iff

   {\rm (bmp)}\quad for each $0$-neighborhood $V$ in $F$ there is a $0$-neighborhood $U$ in $F$ such that whenever $(z_i: i\in I)$ is a finite family of elements of $F$ such that $\sum_{i\in J}z_i\in U$ for all $J\sbs I$, then $\sum_{i\in I} a_i z_i\in V$ when all $|a_i|\le 1$.
   \end{Fact}
   \begin{Rem}
   A closer look at condition (bmp) reveals that it is necessary and sufficient for $Z$ to have the following property:
   For every family $(z_i:i\in I)$ in $Z$, if it satisfies the Cauchy condition for summability, also the family $(a_i z_i:i\in I)$, where $|a_i|\le 1$ for all $i$ satisfies that condition.
   \end{Rem}
   Let's see now, what condition (bmp) gives when applied to vector measures.
    Let $(E,\SE)$ be a measurable space, and let $\bbm:\SE\to F$ be a bounded finitely additive  vector measure. Let $V$ and $U$ be as in (bmp). Then for some $\la>0$ one has $\la\bbm(\SE)\sbs U$. In consequence, by (bmp), whenever $(A_i: i\in I)$ is a finite disjoint family in $\SE$ and $|a_i|\le 1$ ($i\in I$), then $\la\sum_{i\in I}a_i\bbm(A_i)\in V$. In other terms,
   the set $\{\int_E f\,d\bbm \mid f:E\to F \  \text{is $\SE$-simple and $\|f\|_\oo\le 1$}\}$ is bounded in $F$. Thus, in the terminology of Rolewicz \cite[p. 135]{Rol}, $\bbm$ is $L_\oo$-bounded.
      Continuing, consider any finite sequence $A_1,\dots, A_n$ in $\SE$, and nonnegative reals $\alpha_1,\dots,\alpha_n$ with $\sum_i \alpha_i=1$. For $k\le n$ denote: $b_k=\sum_{i\le k}\alpha_i$ and $B_k=A_k\sm (A_1\cup\dots\cup A_{k-1})$. Notice that $B_k$'s are pairwise disjoint and that $A_k=B_1\cup\dots\cup B_k$ for each $k$. In addition, $\sum_i \alpha_i\bbm(A_i)=\sum_k b_k \bbm(B_k)$. It follows that $\conv\bbm(\SE)$ is bounded, and thus $\bbm$ is \textit{convexly bounded} in the alternative terminology of \cite{LDIL-3}. Finishing this part let us stress that the most interesting case of $\bbm$ for us is that of  countably additive vector measures $\bbm:\SE\to L_0(\mu,Z)$
      which, generalizing results of \cite{T} and \cite{KPR}, were shown in \cite{LD-4} to be bounded. The story is not closed yet, for here pops up another nagging question:
      \begin{Pro}
     Is it true that when $\mu$ is a finite positive measure or an o.c. submeasure and $Z$ is an $F$-space such that every countably additive vector measure with values in $Z$ has a bounded range, so is in the case of like measures with values in $L_0(\mu,Z)$?
     \end{Pro}

\subsection*{Acknowledgment}
We are grateful to Iwo Labuda for his interest in this work and many useful suggestions. The second author is especially grateful to Prof. Labuda for his advice regarding the continuity of this work after the passing of Lech Drewnowski.

\appendix 
\section{\\Some Elements of Submeasure Control for Non-Locally Convex Spaces} \label{appendix}

For the convenience of the reader, and in order to keep the proof of the Rybakov theorem completely self-contained, we record here the precise form of the combination argument used above, and based on some elements of \cite{LD-trs3}. Throughout this appendix, $(\mu_t:t\in T)$ denotes a family of countably additive measures from $\SE$ to $L_0(\mu)$, and $\ovl{\mu_t}$ denotes the (classical) semivariation considered in \cite[Sec.\ 10]{LD-trs3}.

We shall use repeatedly the elementary fact, proved in \cite[Sec.\ 10]{LD-trs3}, that for every finitely additive $L_0(\mu)$-valued measure $\nu$ one has
$$
\SN(\ovl\nu)=\{A\in\SE: \nu(B)=0\text{ for every }B\in\SE,\ B\sbs A\}.
$$
In particular,
$$
\ovl\nu(A)=0\iff \nu(B)=0\quad(B\in\SE,\ B\sbs A).
$$

\begin{Lem}\label{lem:small-scalars-app}
Let $X$ be a topological vector space and let $B\sbs X$ be bounded. Then, for every $0$-neighborhood $V$ in $X$, there exists $\ro>0$ such that
$$
|a|<\ro\Longrightarrow ax\in V\qquad(x\in B).
$$
\end{Lem}
\begin{proof}
Since scalar multiplication $\K\times X\to X$ is continuous at $(0,0)$, there exist a neighborhood $I$ of $0$ in $\K$ and a $0$-neighborhood $W$ in $X$ such that
$$
b\in I,\ x\in W\Longrightarrow bx\in V.
$$
Choose $\ro_0>0$ such that $\{b\in\K:|b|<\ro_0\}\sbs I$. As $B$ is bounded, there exists $s>0$ such that $B\sbs sW$. Put $\ro:=\ro_0/s$. If $|a|<\ro$ and $x\in B$, write $x=sw$ with $w\in W$. Then $|as|<\ro_0$, hence $as\in I$, and therefore
$$
ax=(as)w\in V.
$$
This proves the lemma.
\end{proof}

\begin{Lem}\label{lem:omega-domination-app}
Let $\sigma=\sum_{j=1}^m \alpha_j\mu_{t_j}$ be a finite linear combination of members of the family $(\mu_t)_{t\in T}$. Set
$$
q:=\sup_{t\in T}\ovl{\mu_t}.
$$
Then there exists a nondecreasing function $\om:[0,\infty)\to[0,\infty)$ with $\om(r)\to 0$ as $r\downarrow 0$ such that
$$
\ovl\sigma(A)\le \om\bigl(q(A)\bigr)\qquad(A\in\SE).
$$
In particular, if $q$ is order continuous, then $\ovl\sigma$ is order continuous as well.
\end{Lem}
\begin{proof}
Consider the continuous map $\Phi:(L_0(\mu))^m\to L_0(\mu)$ given by
$$
\Phi(x_1,\dots,x_m):=\sum_{j=1}^m \alpha_jx_j.
$$
For $r\ge 0$, let
$$
K_r:=\{(x_1,\dots,x_m)\in (L_0(\mu))^m: |x_j|_\mu\le r\text{ for all }j\}.
$$
Define
$$
\om(r):=\sup\{|\Phi(x)|_\mu:x\in K_r\}.
$$
Since $K_r$ is bounded in $(L_0(\mu))^m$ and $\Phi$ is continuous, the set $\Phi(K_r)$ is bounded in $L_0(\mu)$, hence $\om(r)<\infty$ for every $r$. Clearly $\om$ is nondecreasing. Moreover, continuity of $\Phi$ at $0$ implies $\om(r)\to 0$ as $r\downarrow 0$.

Fix $A\in\SE$ and $B\sbs A$. Then for each $j$,
$$
|\mu_{t_j}(B)|_\mu\le \ovl{\mu_{t_j}}(B)\le q(B)\le q(A).
$$
Hence $(\mu_{t_1}(B),\dots,\mu_{t_m}(B))\in K_{q(A)}$, and therefore
$$
|\sigma(B)|_\mu=\left|\sum_{j=1}^m \alpha_j\mu_{t_j}(B)\right|_\mu\le \om(q(A)).
$$
Taking the supremum over all measurable $B\sbs A$, and using again the collapse property of semivariation, we get
$$
\ovl\sigma(A)\le \om(q(A)).
$$
The last assertion follows immediately: if $A_n\downarrow\varnothing$ and $q(A_n)\to 0$, then
$$
\ovl\sigma(A_n)\le \om(q(A_n))\to 0.
$$
\end{proof}

\begin{Lem}\label{lem:bad-scalars}
Let $\nu_1,\nu_2:\SE\to L_0(\mu)$ be countably additive. Then there exists a countable set $C\sbs \K$ such that, for every $a\in \K\sm C$,
$$
\SN(\ovl{\nu_1+a\nu_2})=\SN(\ovl{\nu_1})\cap \SN(\ovl{\nu_2}).
$$
\end{Lem}
\begin{proof}
Set
$$
\SI:=\SN(\ovl{\nu_1})\cap \SN(\ovl{\nu_2}).
$$
We first note that for every scalar $a\in\K$ one always has
$$
\SI\sbs \SN(\ovl{\nu_1+a\nu_2}).
$$
Indeed, let $A\in\SI$ and let $B\sbs A$ be measurable. Since $A\in\SN(\ovl{\nu_i})$, the collapse property gives $\nu_i(B)=0$ for $i=1,2$. Hence
$$
(\nu_1+a\nu_2)(B)=\nu_1(B)+a\nu_2(B)=0,
$$
so again by the collapse property, $A\in\SN(\ovl{\nu_1+a\nu_2})$.

Thus only the reverse inclusion has to be proved outside a countable set of scalars.
Define
$$
p(A):=\ovl{\nu_1}(A)+\ovl{\nu_2}(A)\qquad(A\in\SE).
$$
Then $p$ is a submeasure. We claim that it is exhaustive on pairwise disjoint sequences. Let $(A_n)$ be pairwise disjoint. Since $\nu_1$ is countably additive, it is exhaustive, and therefore $\nu_1(B_n)\to 0$ for every choice of measurable $B_n\sbs A_n$. If $\ovl{\nu_1}(A_n)\not\to 0$, then for some $\ep>0$ and infinitely many $n$ there would exist $B_n\sbs A_n$ such that $|\nu_1(B_n)|_\mu\ge \ep$, contradicting exhaustivity of $\nu_1$. Hence $\ovl{\nu_1}(A_n)\to 0$. The same argument gives $\ovl{\nu_2}(A_n)\to 0$, and therefore
$$
p(A_n)=\ovl{\nu_1}(A_n)+\ovl{\nu_2}(A_n)\to 0.
$$

Call a scalar $a\in\K$ \emph{bad} if
$$
\SN(\ovl{\nu_1+a\nu_2})\not\sbs \SI.
$$
For each bad scalar $a$ choose a measurable set $E_a\in\SE$ such that
$$
E_a\in \SN(\ovl{\nu_1+a\nu_2})\sm \SI.
$$
Since $E_a\notin \SI$, at least one of $\ovl{\nu_1}(E_a)$ and $\ovl{\nu_2}(E_a)$ is strictly positive; equivalently,
$$
p(E_a)>0.
$$

Now let $a\ne b$ be bad and set $E:=E_a\cap E_b$. We claim that $E\in\SI$. To see this, let $B\sbs E$ be measurable. Since $E_a\in\SN(\ovl{\nu_1+a\nu_2})$ and $E_b\in\SN(\ovl{\nu_1+b\nu_2})$, we have
$$
\nu_1(B)+a\nu_2(B)=0,\qquad \nu_1(B)+b\nu_2(B)=0.
$$
Subtracting these equalities yields $(a-b)\nu_2(B)=0$, and because $a\ne b$, we obtain $\nu_2(B)=0$. Substituting back, we also get $\nu_1(B)=0$. Hence $E\in\SI$.

For $m\in\N$, put
$$
B_m:=\{a\in\K:a\text{ is bad and }p(E_a)\ge 1/m\}.
$$
Then the set $B$ of all bad scalars is equal to $\bigcup_{m=1}^{\infty}B_m$. We show that each $B_m$ is finite. Suppose, on the contrary, that some $B_m$ is infinite. Choose distinct scalars $a_1,a_2,\dots\in B_m$, and define
$$
F_1:=E_{a_1},\qquad F_n:=E_{a_n}\sm \bigcup_{j<n}E_{a_j}\qquad(n\ge 2).
$$
Then $(F_n)$ is a pairwise disjoint sequence, with $F_n\sbs E_{a_n}$ for all $n$. Moreover,
$$
E_{a_n}\cap \bigcup_{j<n}E_{a_j}=\bigcup_{j<n}(E_{a_n}\cap E_{a_j})
$$
involves only sets belonging to $\SI$, and on every set from $\SI$ both semivariations vanish. Hence
$$
p\left(E_{a_n}\cap \bigcup_{j<n}E_{a_j}\right)=0.
$$
By subadditivity of $p$,
$$
p(E_{a_n})\le p(F_n)+p\left(E_{a_n}\cap \bigcup_{j<n}E_{a_j}\right)=p(F_n).
$$
Therefore
$$
p(F_n)\ge p(E_{a_n})\ge 1/m\qquad(n\in\N),
$$
which contradicts the exhaustivity of $p$ on the disjoint sequence $(F_n)$. Thus each $B_m$ is finite, whence $B$ is countable.

If $a\notin B$, then by definition
$$
\SN(\ovl{\nu_1+a\nu_2})\sbs \SI.
$$
Combining this with the inclusion proved at the beginning, we obtain
$$
\SN(\ovl{\nu_1+a\nu_2})=\SI=\SN(\ovl{\nu_1})\cap \SN(\ovl{\nu_2}).
$$
This completes the proof.
\end{proof}

\begin{Lem}\label{lem:small-small}
Let $\eta$ and $q$ be order continuous submeasures on $\SE$, and let $A\in\SE$ satisfy $\eta(A)=0$. Assume that
$$
E\cap A=\varnothing\ \text{and}\ \eta(E)=0\quad\Longrightarrow\quad q(E)=0.
$$
Then, for every $\ep>0$, there exists $\dl>0$ such that
$$
E\cap A=\varnothing\ \text{and}\ \eta(E)\le \dl\quad\Longrightarrow\quad q(E)\le \ep.
$$
\end{Lem}
\begin{proof}
Fix $\ep>0$ and suppose, towards a contradiction, that no such $\dl$ exists. Then for every $n\in\N$ there is a set $E_n\in\SE$ such that
$$
E_n\cap A=\varnothing,\qquad \eta(E_n)\le 2^{-n},\qquad q(E_n)>\ep.
$$
Define
$$
F_n:=\bigcup_{k\ge n}E_k\qquad(n\in\N).
$$
Then $(F_n)$ is a decreasing sequence of measurable sets, each disjoint from $A$, and
$$
\eta(F_n)\le \sum_{k\ge n}\eta(E_k)\le \sum_{k\ge n}2^{-k}=2^{-n+1}.
$$
Hence $\eta(F_n)\to 0$. Put
$$
F:=\bigcap_{n=1}^{\infty}F_n.
$$
Since $\eta$ is order continuous, we get $\eta(F)=0$, and of course $F\cap A=\varnothing$. By the hypothesis of the lemma, this implies $q(F)=0$.

Now fix $n$. Because $E_n\sbs F_n=F\cup(F_n\sm F)$, subadditivity yields
$$
q(E_n)\le q(F)+q(F_n\sm F)=q(F_n\sm F).
$$
But $F_n\sm F\downarrow\varnothing$ as $n\to\infty$, so order continuity of $q$ gives
$$
q(F_n\sm F)\to 0.
$$
Consequently $q(E_n)\to 0$, contradicting the fact that $q(E_n)>\ep$ for every $n$. This contradiction proves the lemma.
\end{proof}

\begin{Th}\label{thm:drewnowski-L0}
Let $(\mu_t:t\in T)$ be a family of countably additive measures from $\SE$ to $L_0(\mu)$ such that the family of semivariations $(\ovl{\mu_t}:t\in T)$ is uniformly order continuous. Assume, in addition, that every countably additive $L_0(\mu)$-valued measure has bounded range. Then there exist a sequence $(t_n)\sbs T$ and scalars $(c_n)\in\ell_1$ such that the series
$$
\la:=\sum_{n=1}^{\infty} c_n\mu_{t_n}
$$
converges uniformly on $\SE$, defines a countably additive measure $\la:\SE\to L_0(\mu)$, and satisfies
$$
\SN(\ovl\la)=\SN\left(\sup_{t\in T}\ovl{\mu_t}\right).
$$
\end{Th}
\begin{proof}
Set
$$
q:=\sup_{t\in T}\ovl{\mu_t}.
$$
By Proposition \ref{prop:countable-uoc}, there exists a countable subset $\{t_n:n\in\N\}\sbs T$ such that
$$
\SN(q)=\SN\Bigl(\sup_{n\in\N}\ovl{\mu_{t_n}}\Bigr).
$$
Replacing the original family by this countable subfamily, we may and shall assume from the outset that $T=\N$.

For $n\in\N$, put
$$
q_n:=\sup_{k\le n}\ovl{\mu_k}.
$$
By \cite[Th.~10.5]{LD-trs3}, applied to the increasing sequence $(q_n)$, there are finite sets $\tau_n\sbs \N$ and a decreasing sequence $(A_n)\sbs \SE$ such that
$$
q_{\tau_n}(A_n)=0,
$$
and
$$
E\cap A_n=\varnothing\ \text{and}\ q_{\tau_n}(E)=0\Longrightarrow q(E)=0,
$$
where
$$
q_{\tau_n}:=\sup_{k\in\tau_n}\ovl{\mu_k},
$$
and, moreover,
$$
q\left(\bigcap_{n=1}^{\infty}A_n\right)=0.
$$

We now construct inductively finitely supported linear combinations $\ga_n=\sum_{k\le n}a_{n,k}\mu_k$ and positive numbers $\dl_n$ such that for every $n\in\N$,
\begin{align*}
&\SN(\ovl{\ga_n})=\bigcap_{k\le n}\SN(\ovl{\mu_k}),\\
&\sup_{E\in\SE}|b_n\mu_{n+1}(E)|_\mu\le 2^{-n-2}\dl_n,\\
&E\cap A_n=\varnothing\ \text{and}\ \ovl{\ga_n}(E)\le \dl_n\Longrightarrow q(E)\le 2^{-n}.
\end{align*}

For $n=1$, set $\ga_1:=\mu_1$. Since $\ga_1$ is a finite linear combination of members of the family, Lemma \ref{lem:omega-domination-app} shows that $\ovl{\ga_1}$ is order continuous. Also,
$$
\SN(\ovl{\ga_1})=\SN(\ovl{\mu_1}).
$$
Because $E\cap A_1=\varnothing$ and $\ovl{\ga_1}(E)=0$ imply $q_{\tau_1}(E)=0$ after a harmless renumbering of the countable family at the start of the proof, the implication furnished by \cite[Th.~10.5]{LD-trs3} yields
$$
E\cap A_1=\varnothing\ \text{and}\ \ovl{\ga_1}(E)=0\Longrightarrow q(E)=0.
$$
Applying Lemma \ref{lem:small-small} with $\eta=\ovl{\ga_1}$ and $A=A_1$, we obtain $\dl_1>0$ such that
$$
E\cap A_1=\varnothing\ \text{and}\ \ovl{\ga_1}(E)\le \dl_1\Longrightarrow q(E)\le 2^{-1}.
$$

Assume now that $\ga_n$ and $\dl_n$ have been constructed. Since $\mu_{n+1}$ is countably additive and $L_0(\mu)$-valued, its range is bounded. Applying Lemma \ref{lem:small-scalars-app} in the topological vector space $L_0(\mu)$ to the bounded set $\mu_{n+1}(\SE)$ and to the neighborhood
$$
V:=\{x\in L_0(\mu):|x|_\mu<2^{-n-2}\dl_n\},
$$
we obtain $\ro_n>0$ such that
$$
|b|<\ro_n\Longrightarrow \sup_{E\in\SE}|b\mu_{n+1}(E)|_\mu\le 2^{-n-2}\dl_n.
$$
Let $C_n$ be the countable exceptional set furnished by Lemma \ref{lem:bad-scalars} for the pair $(\ga_n,\mu_{n+1})$. Choose $b_n\in\K$ with
$$
0<|b_n|<\ro_n,\qquad b_n\notin C_n,
$$
and set
$$
\ga_{n+1}:=\ga_n+b_n\mu_{n+1}.
$$
Then, by Lemma \ref{lem:bad-scalars},
$$
\SN(\ovl{\ga_{n+1}})=\SN(\ovl{\ga_n})\cap \SN(\ovl{\mu_{n+1}})=\bigcap_{k\le n+1}\SN(\ovl{\mu_k}).
$$
Moreover, $\ga_{n+1}$ is again a finite linear combination of members of the family, so Lemma \ref{lem:omega-domination-app} implies that $\ovl{\ga_{n+1}}$ is order continuous.

We now verify the zero-zero implication needed for Lemma \ref{lem:small-small}. Let $E\in\SE$ satisfy
$$
E\cap A_{n+1}=\varnothing,\qquad \ovl{\ga_{n+1}}(E)=0.
$$
Then $E\in\SN(\ovl{\ga_{n+1}})$, hence, by the identity of null ideals just obtained,
$$
E\in\bigcap_{k\le n+1}\SN(\ovl{\mu_k}).
$$
Therefore $q_{\tau_{n+1}}(E)=0$, because every $k\in\tau_{n+1}$ is one of the first finitely many indices under consideration. Since also $E\cap A_{n+1}=\varnothing$, the property of $A_{n+1}$ obtained from \cite[Th.~10.5]{LD-trs3} gives $q(E)=0$.

Applying Lemma \ref{lem:small-small} to $\eta=\ovl{\ga_{n+1}}$, $A=A_{n+1}$ and $q$, we obtain $\dl_{n+1}>0$ such that
$$
E\cap A_{n+1}=\varnothing\ \text{and}\ \ovl{\ga_{n+1}}(E)\le \dl_{n+1}\Longrightarrow q(E)\le 2^{-n-1}.
$$
This completes the inductive construction.

Next we define
$$
\la:=\sum_{n=1}^{\infty} c_n\mu_n,
$$
where $c_1:=1$, and for $n\ge 2$, $c_n$ is the coefficient of $\mu_n$ appearing when the recursive definition of the $\ga_k$'s is unwound. Equivalently, one may write
$$
\ga_{n+1}=\mu_1+\sum_{k=1}^{n}b_k\mu_{k+1},
$$
so that $c_{n+1}=b_n$ for $n\ge 1$.

For $m<n$ and $E\in\SE$ we have
$$
\left|\sum_{k=m+1}^{n}c_k\mu_k(E)\right|_\mu\le \sum_{k=m+1}^{n}\sup_{F\in\SE}|c_k\mu_k(F)|_\mu.
$$
By the choice of the coefficients,
$$
\sup_{F\in\SE}|c_{k+1}\mu_{k+1}(F)|_\mu\le 2^{-k-2}\dl_k,
$$
so the series of suprema converges. Therefore the series defining $\la$ converges uniformly on $\SE$. Since $L_0(\mu)$ is complete, $\la$ is well defined; and because it is the uniform limit of countably additive measures, it is itself countably additive (see the related Nikodym convergence results for F-space valued measures in \cite{LD-trs3}).

We now prove the equality of null ideals. Suppose first that $q(E)=0$. Then $\ovl{\mu_n}(E)=0$ for all $n$, so every finite linear combination of the $\mu_n$ vanishes on every measurable subset of $E$. Passing to the uniform limit, we obtain $\la(B)=0$ for every measurable $B\sbs E$, and hence $\ovl\la(E)=0$. Thus
$$
\SN(q)\sbs \SN(\ovl\la).
$$

Conversely, assume that $\ovl\la(E)=0$. Put
$$
A:=\bigcap_{n=1}^{\infty}A_n.
$$
Since $q(A)=0$, replacing $E$ by $E\sm A$ if necessary, we may suppose that $E\cap A=\varnothing$. Because $A_n\downarrow A$, there exists $m$ such that
$$
q(E\sm A_m)>2^{-m}.
$$
Indeed, otherwise $q(E\sm A_m)\le 2^{-m}$ for every $m$, and by passing to the limit we would get $q(E)=0$.

By the defining property of $\dl_m$, the inequality above implies
$$
\ovl{\ga_m}(E\sm A_m)>\dl_m.
$$
Hence, by the collapse property of semivariation, there exists $F\sbs E\sm A_m$ such that
$$
|\ga_m(F)|_\mu>\dl_m.
$$
On the other hand,
$$
\left|\la(F)-\ga_m(F)\right|_\mu\le \sum_{n>m}\sup_{G\in\SE}|c_n\mu_n(G)|_\mu.
$$
The tail on the right can be made $<\dl_m$ by uniform convergence of the defining series. Therefore $\la(F)\ne 0$. Since $F\sbs E$, this contradicts the assumption $\ovl\la(E)=0$. We conclude that $q(E)=0$.

Thus
$$
\SN(\ovl\la)=\SN(q)=\SN\left(\sup_{t\in T}\ovl{\mu_t}\right),
$$
as claimed.
\end{proof}

\end{document}